\documentclass[oneside,12pt]{amsart}

\usepackage{amsmath,amssymb,amsthm,amsfonts,amstext,amsbsy,amscd,mathrsfs,graphics,graphicx}
\usepackage{amsxtra}
\usepackage[english]{babel}

\usepackage[T1]{fontenc}
\usepackage{latexsym}
\usepackage[latin1]{inputenc}
\usepackage{color}

\newtheorem{theorem}{Theorem}[section]
\newtheorem{lemma}[theorem]{Lemma}

\newtheorem{corollary}[theorem]{Corollary}
\newtheorem{proposition}[theorem]{Proposition}
\newtheorem{example}[theorem]{Example}
\newtheorem{remark}[theorem]{Remark}

\newtheorem{definition}[theorem]{Definition}
\def\bit{\begin{itemize}}
\def\eit{\end{itemize}}
\reversemarginpar   
\def\bc{\begin{center}}
\def\ec{\end{center}}
\def\bthm{\begin{theorem}}
\def\ethm{\end{theorem}}
\def\bcor{\begin{corollary}}
\def\ecor{\end{corollary}}
\def\bprop{\begin{proposition}}
\def\eprop{\end{proposition}}
\def\blem{\begin{lemma}}
\def\elem{\end{lemma}}

\def\brem{\begin{remark}}
\def\erem{\end{remark}}
\def\prf{\noindent{\bf Proof~: }}
\def\bdes{\begin{description}}
\def\edes{\end{description}}

\def\beq{\begin{equation}}
\def\eeq{\end{equation}}

\def\ben{\begin{enumerate}}
\def\een{\end{enumerate}}

\def\beqar{\begin{eqnarray}}
\def\eeqar{\end{eqnarray}}
\def\beqarr{\begin{eqnarray*}}
\def\eeqarr{\end{eqnarray*}}

\def\RR{{\mathbb R}}  

\def\CC{{\mathbb C}}
\def\DD{{\mathbb D}}

\def\EE{{\mathbb E}}
\def\PP{{\mathbb P}}
\def\QQ{{\mathbb Q}}
\def\cA{\mathcal{A}} \def\cB{\mathcal{B}} 
 \def\cE{\mathcal{E}} \def\cF{\mathcal{F}}

\def\cP{\mathcal{P}}

\def\P{{\mathsf P}} 

\def\ZZ{{\mathbb Z}}       
\def\NN{{\mathbb N}}       

\def\p{\varphi}

\def\part{\partial}

\def\d#1dt{\frac{d#1}{dt}}    

\begin{document}
\title{Stochastic flows on metric graphs}

\maketitle
\begin{center}
\renewcommand{\thefootnote}{(\arabic{footnote})}
  \scshape Hatem Hajri\footnote{Universit\'e du Luxembourg, Email: Hatem.Hajri@uni.lu\newline Research supported by the National Research Fund, Luxembourg, and cofunded under the Marie Curie Actions of the European Comission (FP7-COFUND).}
  and Olivier Raimond\footnote{Universit\'e Paris Ouest Nanterre La D\'efense, Email: oraimond@u-paris10.fr}
\renewcommand{\thefootnote}{\arabic{footnote}}\setcounter{footnote}{0}
\end{center}

\section*{Abstract}
We study  a simple stochastic differential equation (SDE) driven by one Brownian motion on a general oriented metric graph whose solutions are stochastic flows of kernels. Under some condition, we describe the laws of all solutions. This work is a natural continuation of \cite{MR2235172}, \cite{MR50101010} and \cite{MR50101111} where some particular metric graphs are considered. 
\section{Introduction}
A metric graph is seen as a metric space with branching points. In recent years, diffusion processes on metric graphs are more and more studied \cite{MR1743222},\cite{MR2905755},\cite{MR2905744},\cite{MR2905733},\cite{MR2905788}. They arise in many physical situations such as electrical networks, nerve impulsion propagation \cite{MR0043202}, \cite{MR0047702}. They also occur in limiting theorems for processes evolving in narrow tubes \cite{MR1399081}. Diffusion processes on graphs are defined in terms of their infinitesimal operators in \cite{MR1743769}. Such processes can be described as mixtures of motions ''along an edge'' and ''around a vertex''.
A typical example of such processes is Walsh Brownian motion defined on a finite number of half lines which are glued together at a unique end point. This process has acquired a particular interest since it was proved by Tsirelson that it can not be a strong solution to any SDE driven by a standard Brownian motion, although it satisfies the martingale representation property with respect to some Brownian motion \cite{MR1022917}. In view of this, it is natural to investigate SDEs on graphs driven by one Brownian motion to be as simple as possible. This study has been initiated by Freidlin and Sheu in \cite{MR1743769} where Walsh Brownian motion has been shown to satisfy the equation
$$df(X_t)=f'(X_t)dW_t+\frac{1}{2} f''(X_t)dt$$
where $W_t=|X_t|-L_t(|X|)$ is a Brownian motion, $f$ runs over an appropriate domain of functions with an appropriate definition of its derivative. Our subject in this paper is to investigate the following extension on a general oriented metric graph: 
$$K_{s,t}f(x) = f(x) +  \int_s^t K_{s,u}f'(x) dW_u + \frac{1}{2}\int_s^t K_{s,u}f''(x)du$$
where $K$ is a stochastic flow of kernels as defined in \cite{MR2060298}, $W$ is a real white noise, $f$ runs over an appropriate domain and $f'$ is defined according to an arbitrary choice of coordinates on each edge. When $G$ is a star graph, this equation has been studied in \cite{MR50101010} and when $G$ consists of only two edges and two vertices the same equation has been considered in \cite{MR50101111}. In this paper, we extend these two studies (as well as \cite{MR2235172} where the associated graph is simply the real line) and classify the solutions on any oriented metric graph.\\ \\ 
The content of this paper is as follows. 

In Section 2, we introduce notations for any metric graph $G$ and then define the SDE $(E)$ driven by  a white noise $W$, with solutions of this SDE being stochastic flows of kernels on $G$. 
Thereafter, our main result is stated. Along an edge the motion of any solution only depends on $W$ and the orientation of the edge. The set of vertices of $G$ will be denoted $V$.
Around a vertex $v\in V$, the motion depends on a flow $\hat{K}^v$ on a star graph (associated to $v$) as constructed in \cite{MR50101010}. 

In Section 3, starting from $(\hat{K}^v)_{v\in V}$ respectively solutions to an SDE on a star graph associated to a vertex $v$, under the following additional (but natural) assumption : the family $\big(\vee_{v\in V} \cF_{s,t}^{\hat{K}^v};\;s\leq t\big)$ is independent on disjoint time intervals, we construct a stochastic flow of kernels $K$ solution of $(E)$ (where $\cF_{s,t}^{\hat{K}^v}$ is the sigma-field generated by the increments of $\hat{K}^v$ between $s$ and $t$).

In Section 4, starting from $K$, we recover the flows $(\hat{K}^v)_{v\in V}$.  Actually, in sections 3 and 4, we prove more general results : the SDEs may be driven by different white noises on different edges of $G$.

The main results about flows on star graphs obtained in \cite{MR50101010} are reviewed in Section 5. Thus, as soon as the flows $(\hat{K}^v)_{v\in V}$ can be defined jointly, we have a general construction of a solution $K$ of $(E)$.

In Section 6, we consider two vertices $v_1$ and $v_2$ and under some condition only depending on the ''geometry'' of the star graphs associated to $v_1$ and $v_2$ we show that independence on disjoint time intervals of $\big({\cF}^{\hat{K}^{v_1}}_{s,t}\vee {\cF}^{\hat{K}^{v_2}}_{s,t}, s\le t\big)$ is equivalent to : $\hat{K}^{v_1}$ and $\hat{K}^{v_2}$ are independent given $W$.

Section 7 is an appendix devoted to the skew Brownian flow constructed by Burdzy and Kaspi in \cite{MR2094439}. We will explain how this flow simplifies our construction on graphs such that any vertex has at most two adjacent edges. 

Section 8 is an appendix complement to Section 5, we will review the construction of flows  $\hat{K}^v$ constructed in \cite{MR50101010} with notations in accordance with the content of our paper.

\section{Definitions and main results}
\subsection{Oriented metric graphs}\label{mlk}

\begin{figure}[!htb]
\centering
\includegraphics[height=5cm,width=7cm]{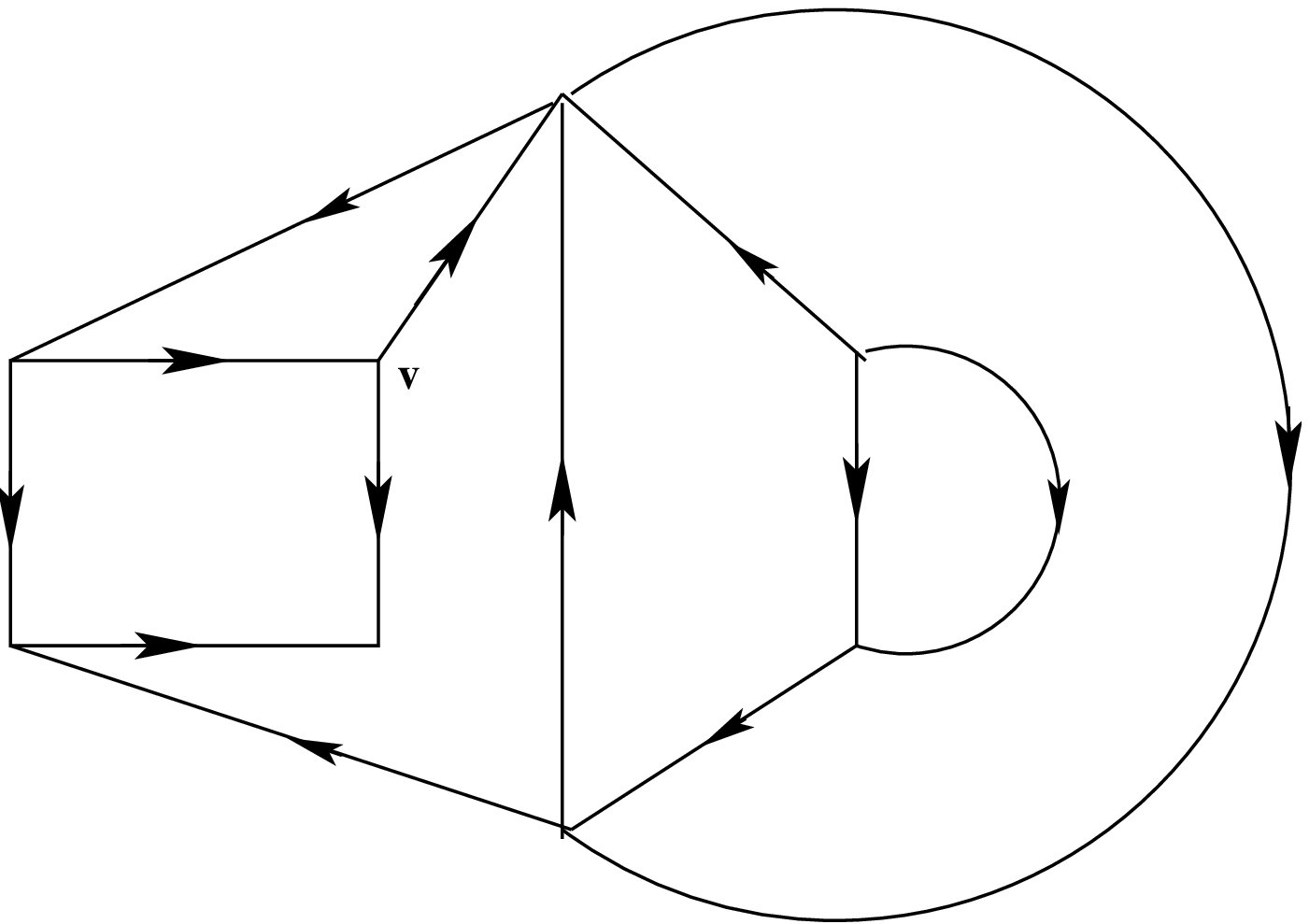}
\caption{An example of oriented metric graph.}
\label{fig:1}
\end{figure}

Let $G$ be a metric graph in the sense that $(G,d)$ is a connected metric space for which there exists a finite or countable set $V$, the set of vertices, and a partition $\{E_i;\;i\in I\}$ of $G\backslash V$ with $I$ a finite or countable set (i.e. $G\backslash V=\cup_{i\in I} E_i$  and for $i\neq j$, $E_i\cap E_j=\emptyset$) such that for all $i\in I$,  $E_i$ is isometric to an interval $(0,L_i)$, with $L_i\le +\infty$. We call $E_i$ an edge, $L_i$ its length and denote by $\{E_i, i\in I\}$ the set of all edges on $G$.\\
To each edge $E_i$, we associate an isometry $e_i:J_i\to \bar{E}_i$, with $J_i=[0,L_i]$ when $L_i<\infty$ and $J_i=[0,\infty)$ or $J_i=(-\infty,0]$ when $L_i=\infty$. 
Note that $e_i(t)\in E_i$ for all $t$ in the interior of $J_i$, $e_i(0)\in V$ and when $L_i<\infty$, $e_i(L_i)\in V$.
The mapping $e_i$ will be called the orientation of the edge $E_i$ and the family $\cE=\{e_i;\;i\in I\}$ defines the orientation of $G$. 
When $L_i<\infty$, denote $\{g_i,d_i\}=\{e_i(0),e_i(L_i)\}$. When $L_i=\infty$, denote $\{g_i,d_i\}=\{e_i(0),\infty\}$ when $J_i=[0,\infty)$ and $\{g_i,d_i\}=\{\infty,e_i(0)\}$ when  $J_i=(-\infty,0]$. 
For all $v\in V$, denote $I^+_v=\{i\in I;\; g_i=v\}$, $I^-_v=\{i\in I;\; d_i=v\}$ and $I_v=I^+_v\cup I^-_v$. 
Let $n_v$, $n^+_v$ and $n^-_v$ denote respectively the numbers of elements in $I_v$, $I_v^+$ and $I_v^-$. Then $n_v=n_v^++n_v^-$.

\smallskip
We will always assume that
\begin{itemize}
\item  $n_v<\infty$ for all $v\in V$ (i.e. $I_v$ is a finite set).
\item  $\inf_i L_i=L>0$. 
\end{itemize}

A graph with only one vertex and such that $L_i=\infty$ for all $i\in I$ will be called a star graph.
It will also be convenient to imbed any star graph in the complex plane $\CC$. Its unique vertex will be denoted  $0$.\\

For each $v\in V$, denote $G_v=\{v\}\cup\cup_{i\in I_v}E_i$ and $G^L_v=\{x\in G;\;d(x,v)<L\}$, which is then a subset of $G_v$. 
Note that $G_v\cap V=\{v\}$. For each $v\in V$, there exists a star graph $\hat{G}_v$ and a mapping $i_v:G_v\to \hat{G}_v$ such that $i_v:G_v\to i_v(G_v)$ is an isometry.
This implies in particular that $i_v(v)=0$ and that $\hat{G}^L_v=\{x\in \hat{G}_v;\;d(0,v)<L\}=i_v({G}^L_v)$.
For each $i\in I_v$, define $\hat{e}^v_i=i_v\circ e_i$. 
Note that $\hat{G}_v$ can be written in the form $\{0\}\cup \cup_{i\in I_v}\hat{E}^v_i$, with $i_v(E_i)\subset\hat{E}^v_i$ and where $(\hat{E}^v_i)_{i\in I_v}$ is the set of edges of $\hat{G}^v$. The mapping $\hat{e}^v_i$ can be extended to an isometry $(-\infty,0]\to \{0\}\cup \hat{E}^v_i$ when $i\in I_v^-$ and to an isometry $[0,+\infty) \to \{0\}\cup \hat{E}^v_i$ when $i\in I_v^+$.

\begin{figure}[!htb]
\centering
\includegraphics[height=4cm,width=4cm]{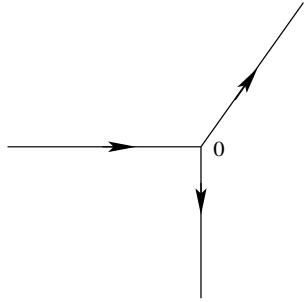}
\caption{The star graph $\hat{G}_v$ associated to $v$ in Figure 1.}
\end{figure}
For $x\in G_v$ and $f:G_v\to\RR$, we will sometimes denote $\hat{x}_v=i_v(x)$ and $\hat{f}_v:\hat{G}_v\to\RR$ the mapping defined by $\hat{f}=0$ on $i_v(G_v)^c$ and $\hat{f}_v=f\circ i_v^{-1}$ on $i_v(G_v)$, so that $\hat{f}_v(\hat{x}_v)=f(x)$ for all $x\in G_v$.

\smallskip
We will also denote by $\cB(G)$ the set of Borel sets of $G$ and by $\cP(G)$ the set of Borel probability measures on $G$. Note that a kernel on $G$ is a measurable mapping $k:G\to\cP(G)$. For $x\in G$ and $A\in \cB(G)$, $k(x,A)$ denotes $k(x)(A)$ and the probability measure $k(x)$ will sometimes be denoted $k(x,dy)$. For $f$ a bounded measurable mapping on $G$, $kf(x)$ denotes $\int f(y) k(x,dy)$.
\subsection{SDE on $G$}
Let $G$ be an oriented metric graph.
To each $v\in V$ and $i\in I_v$, we associate a transmission parameter $\alpha^i_v$ such that $\sum_{i\in I_v} \alpha^i_v=1$. Denote $\alpha=(\alpha^i_v;\; v\in V, \;i\in I_v)$. 
Define ${\mathcal D}^G_\alpha$ the set of all continuous functions $f:G\rightarrow\mathbb{R}$ such that for all $i\in I$, $f\circ e_i$ is $C^2$ on the interior of $J_i$ with bounded first and second derivatives both extendable by continuity to $J_i$ and such that for all $v\in V$
$$\sum_{i\in I_v^+} \alpha^i_v \lim_{r\rightarrow 0+}(f\circ e_i)'(r)=\sum_{i\in I_v^-} \alpha^i_v \lim_{r\rightarrow 0-}(f\circ e_i)'(r).$$
Since $\alpha$ will be fixed, ${\mathcal D}^G_\alpha$ will simply be denoted ${\mathcal D}$.
When $\hat{G}_v$ is a star graph as defined before, to the half line $\hat{E}^v_i$, we associate the parameter $\alpha_v^i$ and denote $\alpha_v=(\alpha_v^i;\; i\in I_v)$ and $\hat{\mathcal D}_v=\mathcal{ D}^{\hat{G}_v}_{\alpha_v}$.
For $f\in\mathcal D$ and $x=e_i(r)\in G\backslash V$, set $f'(x)=(f\circ e_i)'(r)$, $f''(x)=(f\circ e_i)''(r)$ and take the convention $f'(v)=f''(v)=0$ for all $v\in V$.

\begin{definition} \label{defsfk} 
A stochastic flow of kernels (SFK) $K$ on $G$, defined on a probability space $(\Omega,\cA,\PP)$, is a family $(K_{s,t})_{s\le t}$ such that
\begin{enumerate}
\item For all $s\le t$, $K_{s,t}$ is a measurable mapping from $(G\times\Omega,\cB(G)\otimes\cA)$ to $(\cP(G),\cB(\cP(G)))$;
\item For all $h\in\RR$, $s\le t$, $K_{s+h,t+h}$ is distributed like $K_{s,t}$;
\item For all $s_1\le t_1\le\cdots \le s_n\le t_n$, the family $\{K_{s_i,t_i}, 1\le i\le n\}$ is independent.
\item For all $s\le t\le u$ and all $x\in G$, a.s. $K_{s,u}(x)=K_{s,t}K_{t,u}(x)$, and $K_{s,s}$ equals the identity;
\item For all $f\in C_0(G)$, and $s\le t$, we have
$$\lim_{(u,v)\to (s,t)}\sup_{x\in G}\EE[(K_{u,v}f(x)-K_{s,t}f(x))^2]=0;$$
\item For all $f\in C_0(G)$, $x\in G$, $s\le t$, we have
$$\lim_{y\to x}\EE[(K_{s,t}f(y)-K_{s,t}f(x))^2]=0;$$
\item For all $s\le t$, $f\in C_0(G)$, $\lim_{|x|\to\infty}\EE[(K_{s,t}f(x))^2]=0$.
\end{enumerate}
\end{definition}
We say that $\varphi$ is a stochastic flow of mappings (SFM) on $G$ if $K_{s,t}(x)=\delta_{\p_{s,t}(x)}$ is a SFK on $G$.

Given two SFK's $K^1$ and $K^2$ on $G$, we say that $K^1$ is a modification of $K^2$ if for all $s\le t$, $x\in G$, a.s. $K^1_{s,t}(x)=K^2_{s,t}(x)$.

For a family of random variables $Z=(Z_{s,t})_{s\le t}$, 
denote $\mathcal F^Z_{s,t}=\sigma(Z_{u,v}, s\leq u\leq v\leq t)$.

\begin{definition}(Real white noise)
A family $(W_{s,t})_{s\le t}$ is called a real white noise if there exists a Brownian motion on the real line  $(W_t)_{t\in\RR}$, that is $(W_t)_{t\geq 0}$ and $(W_{-t})_{t\geq 0}$ are two independent standard Brownian motions such that for all $s\le t$, $W_{s,t}=W_t-W_s$ (in particular, when $t\ge 0$, $W_t=W_{0,t}$ and $W_{-t}=-W_{-t,0}$). 
\end{definition}

Our main interest in this paper is the following SDE, that extends Tanaka's SDE to metric graphs.
\begin{definition}(Equation $(E_\alpha^G)$)
On a probability space $(\Omega,\mathcal A,\PP)$, let $W$ be a real white noise and $K$ be a stochastic flow of kernels on $G$. We say that $(K,W)$ solves $(E_\alpha^G)$ if for all $s\leq t$, $f\in \mathcal D$ and $x\in G$, a.s.
$$K_{s,t}f(x)=f(x)+\int_s^tK_{s,u}f'(x)W(du) + \frac{1}{2}\int_s^tK_{s,u}f''(x)du.$$
When $\varphi$ is a SFM and $K=\delta_\varphi$ is a solution of $(E)$, we simply say that $(\varphi,W)$ solves $(E_\alpha^G)$.
\end{definition}

Since $G$ and $\alpha$ will be fixed from now on, we will denote equation $(E_\alpha^G)$ simply by $(E)$, and we will also denote $(E^{\hat{G}_v}_{\alpha_v})$ simply be $(\hat{E}^{v})$. A complete classification of solutions to $(\hat{E}^{v})$ has been given in \cite{MR50101010}.\\

\medskip
A family of $\sigma$-fields $(\cF_{s,t};\;s\le t)$ will be said {\em independent on disjoint time intervals (abbreviated : i.d.i)} as soon as for all $(s_i,t_i)_{1\le i\le n}$ with $s_i\le t_i\le s_{i+1}$, the $\sigma$-fields $(\cF_{s_i,t_i})_{1\le i\le n}$ are independent.
Note that for $K$ a SFK, since the increments of $K$ are independent, then $(\cF^K_{s,t};\;s\le t)$ is i.d.i.

Our main result is the following

\begin{theorem}\label{klo}
{\bf\em (i)} Let $W$ be a real white noise and let $(\hat{K}^v)_{v\in V}$ be a family of  SFK's respectively on $\hat{G}_v$.
Assume that for each $v\in V$, $(\hat{K}^v,W)$ is a solution of $(\hat{E}^{v})$ and that $\big(\hat{\cF}_{s,t}:=\vee_{v\in V} \cF_{s,t}^{\hat{K}^v};\;s\leq t\big)$ is independent on disjoint time intervals. Then there exists a unique (up to modification) SFK $K$ on $G$ such that 
\begin{itemize}
\item $\mathcal F^K_{s,t}\subset\hat{\cF}_{s,t}$ for all $s\leq t$, 
\item $(K,W)$ is a solution to $(E)$ and 
\item For all $s\in\RR$ and $x\in G_v$, setting
\begin{equation}\label{rho} \rho^{x,v}_s=\inf\{u\geq s : \; K_{s,u}(x,G_v)<1\},\end{equation} 
then for all $t>s$, a.s. on the event $\{t<\rho^{x,v}_s\}$,
\begin{equation}\label{pln}
i_v * K_{s,t}(x) = \hat{K}^v_{s,t}(\hat{x}^v) .
\end{equation}
\end{itemize}

{\em\bf(ii)} Let $(K,W)$ be a solution of $(E)$. Then for each $v\in V$, there exists a unique (up to modification) SFK $\hat{K}^{v}$ on $\hat{G}_v$ such that 
\begin{itemize}
\item for all $s\le t$, $\hat{\cF}_{s,t}:=\vee_{v\in V} \cF_{s,t}^{\hat{K}^v}\subset{\cF}_{s,t}^K$,
\item $(\hat{K}^v,W)$ is a solution of $(\hat{E}^{v})$ for each $v\in V$
\end{itemize}
and such that if $\rho^{x,v}_s$ is defined by (\ref{rho}) then for all $s<t$ in $\RR$ and $x\in G_v$, a.s. on the event $\{t<\rho^{x,v}_s\}$, (\ref{pln}) holds. 
\end{theorem}

Note that \eqref{pln} can be rewritten: 
for all bounded measurable function $f$ on $G$, and all $x\in G_v$
$$K_{s,t}f(x)=\hat{K}_{s,t}\hat{f}^v (\hat{x}^v).$$

Theorem \ref{klo} reduces the construction of solutions to $(E)$ to the construction of solutions to $(\hat{E}^v)$. 
Since for all $v$ all solutions to $(\hat{E}^v)$ are described in \cite{MR50101010}, to complete the construction of all solutions to $(E)$, one has to be able to construct them jointly.

This Theorem implies that there is a unique $\sigma(W)$-measurable flow solving $(E)$.  We also notice that under the assumption $\big(\hat{\cF}_{s,t}; s\leq t\big)$ is i.d.i it is possible to construct different (in law) flows of mappings solving $(E)$. However, assuming that solutions to $(\hat{E}^v)$ are independent given $W$ the associated flow of mappings solution to $(E)$ is law-unique. This applies also to all other solutions.\\
For each $v\in V$, let $\alpha^+_v=\sum_{i\in I_v^+} \alpha_v^i$ and $\beta_v=2\alpha^+_v-1$. Under some condition linking $\beta_{v_1}$ and $\beta_{v_2}$, the next proposition offers a better understanding of : $({\cF}^{\hat{K}^{v_1}}_{s,t}\vee {\cF}^{\hat{K}^{v_2}}_{s,t})_{s\le t}$ is i.d.i. 
\begin{proposition}\label{condi} Let $v_1$ and $v_2$ be two vertices in $V$ such that $\beta_{v_2}\neq \beta_{v_1}$ and 
$$|\beta_{v_2}-\beta_{v_1}|\ge 2\beta_{v_1}\beta_{v_2}.$$
Let $W$ be a real white noise. Let $\hat{K}^{v_1}$ and $\hat{K}^{v_2}$ be SFKs respectively on $\hat{G}^{v_1}$ and on $\hat{G}^{v_2}$ such that $(\hat{K}^{v_1},W)$ and $(\hat{K}^{v_2},W)$ are solutions respectively to $(\hat{E}^{v_1})$ and to  $(\hat{E}^{v_2})$. 
Then $({\cF}^{\hat{K}^{v_1}}_{s,t}\vee {\cF}^{\hat{K}^{v_2}}_{s,t})_{s\le t}$ is i.d.i if and only if 
 $\hat{K}^{v_1}$ and $\hat{K}^{v_2}$ are independent given $W$. 

\end{proposition}
When $V=\{v_1,v_2\}$ with $\hat{G}^{v_1}$ and $\hat{G}^{v_2}$ being given by the following star graphs 

\begin{figure}[!htb]
\centering
\includegraphics[scale=.4]{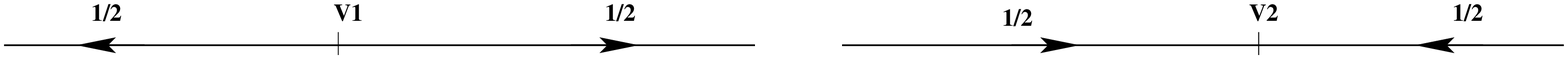}
\caption{$\hat{G}^{v_1}$ and $\hat{G}^{v_2}$.}
\end{figure}

\noindent this proposition has been proved in \cite{MR50101111}.

\section{Construction of a solution of $(E)$ out of solutions of $(\hat{E}^v)$}

For all $i\in I$, let $W^i$ be a real white noise.
Assume that $\mathcal{W}:=(W^i_{s,t};\; i\in I,\, s\le t)$ is Gaussian. 
Let 
\begin{equation}\label{A}
A_{s,t}:=\{\sup_{i\in I}\sup_{s<u<v<t}|W^i_{u,v}|<L\}.
\end{equation}
Assume that $\lim_{|t-s|\to 0} \PP(A_{s,t}^c)=0$. Note that this assumption is satisfied if $W^i=W$ for all $i$, or if $I$ is finite.

Let $\hat{K}=(\hat{K}^v)_{v\in V}$ be a family of SFK's respectively on $\hat{G}_v$ and let $\mathcal{W}^v:=(W^i;\;i\in I_v)$. Assume that $(\hat{K}^v,\mathcal{W}^v)$ is a solution to the following SDE: For all $s\leq t$, $\hat{f}\in \hat{\mathcal D}_v$, $\hat{x}\in \hat{G}_v$, a.s.
\begin{equation}\label{appro}
\hat{K}^v_{s,t}\hat{f}(\hat{x})=\hat{f}(\hat{x})+\sum_{i\in I_v} \int_s^t\hat{K}^v_{s,u}(1_{\hat{E}_i^v}\hat{f}')(\hat{x})W^i(du) + \frac{1}{2}\int_s^t\hat{K}^v_{s,u}\hat{f}''(\hat{x})du.
\end{equation}
\medskip
Then we have the following
\begin{lemma}\label{noise}
For all $v\in V$, $i\in I_v$ and all $s\le t$, we have $\mathcal F^{W^i}_{s,t}\subset\cF_{s,t}^{\hat{K}^v}$.
\end{lemma}
\prf
Let $y=\hat{e}^{v}_i(r)\in \hat{E}^v_i$. Following Lemma 6 \cite{MR50101010}, we prove that $\hat{K}^v_{s,t}(y)=\delta_{\hat{e}^{v}_i(r+W^{i}_{s,t})}$ for all $s\le t\le \sigma^y_s$ where
$$\sigma_s^{y}=\inf\{u\geq s; \; \hat{e}^v_i(r+W^i_{s,u})=0\}.$$
Since this holds for arbitrarily large $r$, the lemma holds.
\qed

\medskip
In all this section, we assume that 
\begin{equation}\label{assum}
\big(\hat{\cF}_{s,t}:=\vee_{v\in V} \cF_{s,t}^{\hat{K}^v};\; s\leq t\big) \quad \text{is i.d.i.} 
\end{equation}
We will prove the following 
\bthm \label{hatktok}
There exists $K$ a unique (up to modification) SFK on $G$, such that 
\begin{itemize}
\item $\mathcal F^K_{s,t}\subset\hat{\cF}_{s,t}$ for all $s\leq t$,
\item $(K,W)$ is a solution to the SDE:
For all $s\leq t$, $f\in {\mathcal D}$, $x\in {G}$, a.s.
$$K_{s,t}f(x)=f(x)+\sum_{i\in I} \int_s^tK_{s,u}(1_{E_i}f')(x)W^i(du) + \frac{1}{2}\int_s^tK_{s,u}f''(x)du.$$
\end{itemize}
and such that defining for $s\in\RR$, $v\in V$ and $x\in G_v$,
\begin{equation}\label{defin}
\rho^{x,v}_s=\inf\{u\geq s : \;K_{s,u}(x,G_v)<1\}
\end{equation}
we have that for all $s<t$ in $\RR$ and $x\in G_v$, a.s. on the event $\{t<\rho^{x,v}_s\}$,
\begin{equation}\label{holf}
i_v * K_{s,t}(x) = \hat{K}^v_{s,t}(\hat{x}^v).
\end{equation}

\ethm

Note that this Theorem implies (i) of Theorem \ref{klo}.

\subsection{Construction of $K$}\label{cons}
For all $s\in\RR$, $x\in G$, define
$$\tau_s^x=\inf\{t\geq s;\; e_i(r+W^i_{s,t})\in V\}$$
where $i\in I$ and $r\in \RR$ are such that $x=e_i(r)$. For $s<t$, define the kernel $K^0_{s,t}$ on $G$ by:
On the event $A_{s,t}^c$, set $K^0_{s,t}(x)=\delta_x$, and on the event $A_{s,t}$, if $x=e_i(r)\in G$ and $v=e_i(r+W^i_{s,\tau_s^x})$, then
\begin{eqnarray*}
K^0_{s,t}(x)&=&\left\{\begin{array}{ccccc} \delta_{e_i(r+W^i_{s,t})} & \hbox{ if } & t\le \tau_s^x&&\\
i_v^{-1}*\hat{K}^v_{s,t}(\hat{x}^v) & \hbox{ if } & t > \tau_s^x  \end{array}\right.
\end{eqnarray*} 
(i.e. for $A\in\cB(G)$, $\big(i_v^{-1}*\hat{K}^v_{s,t}(\hat{x}^v)\big)(A)=\hat{K}^v_{s,t}(\hat{x}^v,i_v(A\cap G_v))$). Note that on $A_{s,t}\cap\{t>\tau_s^x\}\cap\{v=e_i(r+W^i_{s,\tau_s^x})\}$, we have that the support of $\hat{K}^v_{s,t}(\hat{x}^v)$ is included in $i_v(G_v)$ so that $K^0_{s,t}(x)\in\cP(G)$. 
Remark also that on $A_{s,t}\cap\{v=e_i(r+W^i_{s,\tau_s^x})\}$, a.s.
$$K^0_{s,t}(x)=i_v^{-1}*\hat{K}^v_{s,t}(\hat{x}^v).$$
\blem\label{st} For all $s<t<u$ and all $\mu\in \cP(G)$, a.s. on $A_{s,u}$,
\begin{equation}\label{flowK0} \mu K^0_{s,u}=\mu K^0_{s,t}K^0_{t,u}.\end{equation}
\elem
\prf Fix $s<t<u$ and note that $A_{s,u}\subset A_{s,t}\cap A_{t,u}$ a.s. We prove the lemma for $\mu=\delta_x$ which is enough since by Fubini's Theorem : $\forall A\in\mathcal B(G)$
$$\EE[|\mu K^0_{s,u}(A)-\mu K^0_{s,t}K^0_{t,u}(A)|]\le \int_G \EE[|K^0_{s,u}(x,A)-K^0_{s,t}K^0_{t,u}(x,A)|]\mu(dx).$$
There exist $i$ and $r$ such that $x=e_i(r)$. 
Denote by $Y=e_i(r+W^i_{s,t})$, when $t\leq\tau_s^x$.\\ 
If $u\leq\tau_s^x$, then it is easy to see that \eqref{flowK0} holds after having remarked that $\tau_t^Y=\tau_s^x$.\\
If $t\le \tau_s^x<u$, then $K^0_{s,t}(x)=\delta_Y$ and $K^0_{s,u}(x)=i_v^{-1}*\hat{K}^v_{s,u}(\hat{x}^v)$ with $v=e_i(r+W^i_{s,\tau_s^x})$. We still have $\tau_t^Y=\tau_s^x$ which is now less than $u$. 
Write $K^0_{s,t}K^0_{t,u}(x)=K^0_{t,u}(Y)=i_{v'}^{-1}*\hat{K}^{v'}_{t,u}(\hat{Y}^v)$ where $v'=e_i(e_i^{-1}(Y)+W^i_{t,\tau_t^Y})$. Note that $v'=v$ since (we have $e_i^{-1}(Y)=r+W^i_{s,t}$)
$$v'=e_i(r+W^i_{s,t}+W^i_{t,\tau_t^Y})=e_i(r+W^i_{s,\tau_s^x})=v.$$
Since $\hat{K}^v$ is a flow, we get $$K^0_{s,t}K^0_{t,u}(x)=i_v^{-1}*\hat{K}^v_{t,u}(\hat{Y}^v)=i_v^{-1}*\hat{K}^v_{s,t}\hat{K}^v_{t,u}(\hat{x}^v)=i_v^{-1}*\hat{K}^v_{s,u}(\hat{x}^v)=K^0_{s,u}(x).$$
If $\tau_s^x<t$, then $K^0_{s,t}(x)=i_v^{-1}*\hat{K}^v_{s,t}(\hat{x}^v)$ and  $K^0_{s,u}(x)=i_v^{-1}*\hat{K}^v_{s,u}(\hat{x}^v)$ with $v$ defined as above.
Let $f$ be a bounded measurable function on $G$. Then $K^0_{s,u}f(x)=\hat{K}^v_{s,u}\hat{f}^v(\hat{x}^v)$. And since $\hat{K}^v$ is a flow,
$$K^0_{s,u}f(x)=\hat{K}^v_{s,t}\hat{K}^v_{t,u}\hat{f}^v(\hat{x}^v).$$
Note that on the event $A_{s,t}\cap \{\tau_s^x<t\}$, the support of $K^0_{s,t}(x)$ is included in $G^L_v$, and for all $y$ in the support of $K^0_{s,t}(x)$, the support of  $\hat{K}^v_{t,u}(\hat{y}^v)$ is included in $\hat{G}^L_v$. In other words, it holds that on the event $A_{s,t}\cap \{\tau_s^x<t\}$, for all $y$ in the support of $K^0_{s,t}(x)$, $\hat{K}^v_{t,u}\hat{f}^v(\hat{y}^v)=K^0_{t,u}f(y)$ and thus that $\hat{K}^v_{s,t}\hat{K}^v_{t,u}\hat{f}^v(\hat{y}^v)=K^0_{s,t}K^0_{t,u}f(y)$. This implies the Lemma. \qed

\medskip
We will say that a random kernel $K$ is {\em Fellerian} when for all $n\ge 1$ and all $h\in C_0(G^n)$, we have $\EE[K^{\otimes n} h]\in C_0(G^n)$.

\blem \label{k0feller} For all $s<t$ , $K^0_{s,t}$ is Fellerian. \elem
\prf By an approximation argument (see the proof of Proposition 2.1 \cite{MR2060298}), it is enough to prove the following $L^2$-continuity for $K^0$ : for all $f\in C_0(G)$ and all $x\in G$, $\lim_{y\to x} \EE[(K^0_{0,t}f(y)-K^0_{0,t}f(x))^2]=0$. Write
$$(K^0_{0,t}f(y)-K^0_{0,t}f(x))^2=(K^0_{0,t}f(y)-K^0_{0,t}f(x))^2 1_{A_{0,t}}+(f(y)-f(x))^2 1_{A_{0,t}^c}.$$
Suppose that $x$ belongs to an edge $E_i$. Using the convergence in probability $W^i_{\tau_0^y}\rightarrow W^i_{\tau_0^x}$ as $y\rightarrow x$, we see that $\PP(K^0_{0,\tau_0^y}(y)\neq K^0_{0,\tau_0^x}(x))$ converges to $0$ as  $y\to x$. To conclude, it remains to prove that for $v\in\{g_i,d_i\}$ (i.e. $v$ is an end point of $E_i$), denoting  $C^v_t=A_{0,t}\cap\{K^0_{0,\tau_0^y}(y)=K^0_{0,\tau_0^x}(x)=\delta_{v}\}$, we have
$$\lim_{y\to x}\EE[(K^0_{0,t}f(y)-K^0_{0,t}f(x))^2 1_{C^v_t}]=0.$$
Since on $C^v_t$, $K^0_{0,t}(z)=i_{v}^{-1}*\hat{K}^{v}_{0,t}(\hat{z}^{v})$ for $z\in\{x,y\}$, our result holds. \qed

\medskip 
\blem\label{lp}Let $K_1$ and $K_2$ be two independent Fellerian  kernels. Then $K_1K_2$ is a Fellerian kernel. \elem 
\prf Set $\P^{(n)}_1=\EE[K_1^{\otimes n}]$ and $\P^{(n)}_2=\EE[K_2^{\otimes n}]$. Then $\P^{(n)}_1\P^{(n)}_2=\EE[(K_1K_2)^{\otimes n}]$. This implies the lemma. \qed

\medskip

Define for $n\in\NN$, $\DD_n:=\{k2^{-n};\; k\in\ZZ\}$. For $s\in \RR$, let $s_n=\sup\{u\in \DD_n;\; u\leq s\}$ and $s^+_n=s_n+2^{-n}$. For every $n\geq 1$ and $s\leq t$ define 
$$K^n_{s,t}=K^0_{s,s^+_n}K^0_{s^+_n,s^+_n+2^{-n}}\dots K^0_{t_n-2^{-n},t_n}K^0_{t_n,t}.$$
if $s_n^+\leq t$ and $K^n_{s,t}=K^0_{s,t}$ if $s_n^+>t$.
Note that Lemma \ref{k0feller} and Lemma \ref{lp} imply that $K^n_{s,t}$ is Fellerian (since the kernels $K^0_{s,s^+_n}$, $K^0_{s^+_n,s^+_n+2^{-n}}$, $\dots ,K^0_{t_n-2^{-n},t_n}$, $K^0_{t_n,t}$ are independent by (\ref{assum})). \\

Define $\Omega^n_{s,t}=\{\sup_i\sup_{\{s<u<v<t;\, |v-u|\le 2^{-n}\}}|W^i_{u,v}|<L\}$. Note that for all $s\leq u<v\leq t$ such that $|u-v|\leq 2^{-n}$, we have $\Omega^n_{s,t}\subset A_{u,v}$.\\
Let $\Omega_{s,t}=\cup_n\Omega^n_{s,t}$, then $\PP(\Omega_{s,t})=1$. Define now, for $\omega\in \Omega_{s,t}$, $K_{s,t}(\omega)=K^n_{s,t}(\omega)$ where $n=n_{s,t}=\inf\{k;\; \omega\in\Omega^k_{s,t}\}$ and set $K_{s,t}(x)=\delta_{x}$ on $\Omega_{s,t}^c$.

\blem \label{lemomn} For all $s<t$ and all $\mu\in\cP(G)$,  a.s. we have
$$\mu K^m_{s,t}=\mu K_{s,t}\quad \text{for all}\ m\geq n_{s,t}.$$
\elem
\prf For $m\geq n_{s,t}$, we have (denoting $n=n_{s,t}$)
$$\mu K_{s,t}=\mu K^0_{s,s^+_n}K^0_{s^+_n,s^+_n+2^{-n}}\dots K^0_{t_n-2^{-n},t_n}K^0_{t_n,t}.$$
where $s^+_n, s^+_n+2^{-n},\cdots,t_n$ are also in $\DD_m$. Moreover for all $(u,v)\in\{(s,s^+_n), (s^+_n, s^+_n+2^{-n}),\cdots,(t_n,t)\}$, we have $\Omega^n_{s,t}\subset A_{u,v}$. Now applying Lemma \ref{st} and an independence argument, we see that $\mu K_{s,t}=\mu K^m_{s,t}$.
\qed

\bprop $K$ is a SFK. \eprop
\prf Obviously the increments of $K$ are independent. Fix $s<t<u$, then by the previous lemma and Lemma \ref{st} a.s. for $m$ large enough (i.e. $m\ge \max\{n_{s,u},n_{s,t},n_{t,u}\}$), we have
\begin{eqnarray}
\mu K_{s,u}&=&\mu K^m_{s,u}=\mu K^m_{s,t_m}K^m_{t_m,t_m^+}K^{m}_{t^+_m, t^+_m+2^{-m}}\cdots K^m_{u_m,u}\nonumber\\
&=&\mu K^m_{s,t_m}K^m_{t_m,t}K^{m}_{t,t^+_m}K^{m}_{t^+_m, t^+_m+2^{-m}}\cdots K^m_{u_m,u}\nonumber\\
&=&\mu K^{m}_{s,t}K^{m}_{t,u}\nonumber\\
&=&\mu K_{s,t}K_{t,u}.\nonumber\
\end{eqnarray} 
This proves that $K$ satisfies the flow property.

Fix $k\ge 1$, $h\in C^0(G^k)$. Let $\alpha>0$ and $n_1\in\NN$ such that $\mathbb P(n_{s,t}>n_1)<\alpha$. Then for all $x, y\in G^k$, since $\PP(\Omega_{s,t})=1$, we have
\begin{eqnarray}
|\EE[K_{s,t}^{\otimes k}h(y)]-\EE[K_{s,t}^{\otimes k}h(x)]|&\leq& \sum_{n\leq n_1} \EE\big[\big((K^n_{s,t})^{\otimes k}h(y)-(K^n_{s,t})^{\otimes k}h(x)\big)^2\big]^{\frac{1}{2}}\nonumber\\
&+&2\alpha||h||_{\infty}.\nonumber\
\end{eqnarray}
Now since $K^n$ is Feller for all $n$, we deduce that $$\limsup_{y\rightarrow x}|\EE[K_{s,t}^{\otimes k}h(y)]-\EE[K_{s,t}^{\otimes k}h(x)]|\leq 2\alpha||h||_{\infty}.$$
Since $\alpha$ is arbitrary, it holds that for all $s<t$, $K_{s,t}$ is Fellerian.

\blem \label{felent} For all $x\in G$ and $f\in C_0(G)$, $\lim_{|t-s|\to 0} \EE[(K_{s,t}f(x)-f(x))^2]=0$. \elem
\prf  Take $x=e_i(r)$ and let $\epsilon>0$. Then there exists $\alpha>0$ such that $|t-s|<\alpha$ implies $\PP(A_{s,t})>1-\epsilon$. Note that a.s. on $A_{s,t}$, $K_{s,t}(x)=K^0_{s,t}(x)$. If $x\not\in V$, then $\EE[(K_{s,t}f(x)-f(x))^2 1_{A_{s,t}}]\le 2\|f\|_\infty^2 \PP(\tau_s^x<t)+\EE[(f(e_i(r+W^i_{s,t}))-f(e(r)))^2 1_{t\le \tau_s^x}]$. 
The two right hand terms clearly converge to $0$ as $|t-s|$ goes to $0$. This implies the lemma when $x\not\in V$.  
When $x=v\in V$, then a.s. on $A_{s,t}$, $K_{s,t}f(x)=\hat{K}^v_{s,t}\hat{f}^v(0)$. And we can conclude since $\hat{K}^v$ is a SFK.
\qed

\medskip
This lemma with the flow property imply that for all $f\in C_0(G)$ and all $x\in G$, $(s,t)\mapsto K_{s,t}f(x)$ is continuous as a mapping from $\{s<t\}\to L^2(P)$.
Now since for all $s<t$ in $\DD$, the law of $K_{s,t}$ only depends on $|t-s|$, the continuity of this mapping implies that this also holds for all $s<t$. Thus, we have proved that $K$ is a SFK. \qed

\subsection{The SDE satisfied by $K$}\label{sat}
Recall that each flow $\hat{K}^v$ solves equation $(\hat{E}^v)$ defined on $\hat{G}_v$. Then we have
\blem\label{yh}
For all $x\in G, f\in \mathcal D$ and all $s<t$, a.s. on $A_{s,t}$

$$K^0_{s,t}f(x) = f(x) +  \sum_{i\in I} \int_s^t K^0_{s,u}(1_{E_i}f')(x) W^i(du) + \frac{1}{2} \int_s^t K^0_{s,u}f''(x) du.$$
\elem

\prf Let $x=e_i(r)$ with $i\in I_v$. Recall the notation $\hat{x}_v=i_v(x)\in \hat{G}_v$. Then denoting $B_{s,t}^v= A_{s,t}\cap \{\tau_s^x\le t\}\cap\{e_i(r+W^i_{s,\tau_s^x})=v\}$, we have that a.s. on $A_{s,t}$,
\begin{eqnarray*}
K^0_{s,t}f(x)
&=& (f\circ e_i)(r+W^i_{s,t}) 1_{\{\tau_s^x>t\}} +  \sum_{v\in V} \hat{K}^v_{s,t}\hat{f}^v(\hat{x}_v) 1_{B_{s,t}^v}.
\end{eqnarray*}
Thus a.s. on $A_{s,t}$,
\begin{eqnarray*}
K^0_{s,t}f(x)
&=& f(x) \\
&+& 1_{\{\tau_s^x>t\}} \left(\int_s^t(f\circ e_i)'(r+W^i_{s,u}) W^i(du) + \frac{1}{2} \int_s^t(f\circ e_i)''(r+W^i_{s,u})  du \right) \\
&+& \sum_{v\in V} 1_{B_{s,t}^v} \left(\sum_{j\in I_v}\int_s^t \hat{K}^v_{s,u}\big(1_{\hat{E}_j^v}(\hat{f}^v)'\big)(\hat{x}_v) W^j(du) + \int_s^t \hat{K}^v_{s,u}(\hat{f}^v)''(\hat{x}_v) du\right)\\
&=& f(x) + 1_{\{\tau_s^x>t\}} \left(\int_s^t K^0_{s,u}(1_{E_i}f')(x) W^i(du) + \frac{1}{2} \int_s^t K^0_{s,u}f''(x)  du \right) \\
&&  + \sum_{v\in V} 1_{B_{s,t}^v} \left(\sum_{j\in I_v}\int_s^t K^0_{s,u}(1_{E_j}f')(x) W^j(du) + \int_s^t K^0_{s,u}f''(x)  du\right)\\ 
\end{eqnarray*}
This implies the lemma. \qed
\blem
For all $n\in\mathbb N$, $x\in G$, $s<t$ and all $f\in \mathcal D$ a.s. on $\Omega^n_{s,t}$, we have
\begin{eqnarray}
K^n_{s,t}f(x) &=& f(x) +  \sum_{i\in I}\int_s^t K^n_{s,u}(1_{E_i}f')(x) W^i(du) \nonumber\\
&&+ \quad\frac{1}{2}\int_s^t K^n_{s,u}f''(x) du. \nonumber
\end{eqnarray}
\elem
\prf
The proof will be by induction on $q=\text{Card}\ \{s,s^+_n,s^+_n+2^{-n},\cdots,t_n,t\}$. For $q=2$, this is immediate from Lemma \ref{yh} since $\Omega^n_{s,t}\subset A_{s,t}$. Assume this is true for $q-1$ and let $s<t$ such that $\text{Card}\ \{s,s^+_n,s^+_n+2^{-n},\cdots,t_n,t\}=q.$ Then a.s.
\begin{eqnarray}
K^n_{s,t}f(x) &=& K^n_{s,t_n}K^n_{t_n,t}f(x)\nonumber\\
& =& K^n_{s,t_n}\left(f+\sum_{i\in I}\int_{t_n}^{t}  K^n_{t_n,u}(1_{E_i}f') W^i(du) + \frac{1}{2}\int_{t_n}^{t} K^n_{t_n,u}f''du\right)(x)\nonumber\\
&=& K^n_{s,t_n}f(x)\nonumber\\
&+&\sum_{i\in I} \int_{t_n}^{t} K^n_{s,t_n}K^n_{t_n,u}(1_{E_i}f')(x) W^i(du) + \frac{1}{2}\int_{t_n}^{t} K^n_{s,t_n}K^n_{t_n,u}f''(x)du\nonumber\\
&=& f(x) +  \sum_i \int_s^t K^n_{s,u}(1_{E_i}f')(x) W^i(du) + \frac{1}{2}\int_s^t K^n_{s,u}f''(x) du.\nonumber\
\end{eqnarray}
by independence of increments and using the fact that $K^n_{s,t_n}(x)$ is supported by a finite number of points. 
\qed
\medskip

Thus we have
\blem
For all $x\in G$, $f\in \mathcal D$ and all $s<t$, a.s. 
\begin{equation} \label{GSDE}
K_{s,t}f(x) = f(x) + \sum_{i\in I} \int_s^t K_{s,u}(1_{E_i}f')(x) W^i(du) + \frac{1}{2} \int_s^t K_{s,u}f''(x) du.
\end{equation}
\elem
\prf Note that for all $n$, on $\Omega^n_{s,t}$, for all $u\in [s,t]$, a.s. $K_{s,u}(x)=K^n_{s,u}(x)$. Thus a.s. on $\Omega^n_{s,t}$, (\ref{GSDE}) holds in $L^2(\PP)$ and finally a.s. (\ref{GSDE}) holds. \qed

\medskip 
Remark: When $W^i=W$ for all $i$, then $(K,W)$ solves the SDE (E).

\smallskip

This Lemma with the fact that $K$ is a SFK permits to prove that $K$ satisfies the first two conditions of Theorem \ref{hatktok}. 
Note that for all $s\le t$ and all $x\in G$, we have that a.s. on $A_{s,t}$,  $K_{s,t}(x)=K^0_{s,t}(x)$. 
Thus a.s. on $A_{s,t}$, (\ref{holf}) holds. 
Now, we want to prove that a.s. (\ref{holf}) holds on the event $\{t<\rho^{x,v}_s\}$ (note that a.s.  $A_{s,t}\cap\{t>\tau^x_s\}\subset\{\tau^x_s<t<\rho^{x,v}_s\}$). 
By Lemma \ref{lemomn}, a.s.  for all $m\geq n_{s,t}$ such that $s_m^+\le t$, 
$$K_{s,t}(x)=K^{0}_{s,s_m^+}K^0_{s^+_m,s^+_m+2^{-m}}\dots K^0_{t_m-2^{-m},t_m}K^0_{t_m,t}(x).$$
Clearly on $\{t<\rho^{s,v}_s\}$, a.s. $K^0_{s,s_m^+}(x)=i_v^{-1}*\hat{K}^{v}_{s,s_m^+}(\hat{x}^v)$ and for all $y$ in the support of $K^0_{s,s_m^+}(x)$, $K^0_{s^+_m,s^+_m+2^{-m}}(y)=i_v^{-1}*\hat{K}^v_{s^+_m,s^+_m+2^{-m}}(\hat{y}^v)$. Thus, on $\{t<\rho^{s,v}_s\}$, a.s. $K^{0}_{s,s_m^+}K^0_{s^+_m,s^+_m+2^{-m}}(x)=i_v^{-1}*\hat{K}^{v}_{s,s_m^+}\hat{K}^v_{s^+_m,s^+_m+2^{-m}}(\hat{x}^v).$ The same argument shows that on $\{t<\rho^{s,v}_s\}$, a.s.
$$K_{s,t}(x)=i_v^{-1}*\hat{K}^{v}_{s,s_m^+}\hat{K}^v_{s^+_m,s^+_m+2^{-m}}\dots \hat{K}^v_{t_m-2^{-m},t_m}\hat{K}^v_{t_m,t}(\hat{x}^v)=i_v^{-1}*\hat{K}^{v}_{s,t}(\hat{x}^v).$$
To conclude the proof of Theorem \ref{hatktok}, it remains to prove that if $K'$ is a SFK satisfying also the conditions of Theorem \ref{hatktok}, then $K'$ is a modification of $K$. Since (\ref{holf}) holds for $K$ and $K'$, for all $s\le t$ and all $\mu\in\cP(G)$ a.s. on $A_{s,t}$, $\mu K'_{s,t}=\mu K_{s,t}(=\mu K^0_{s,t})$. Thus for all $s\le t$ and $x\in G$, denoting $n=n_{s,t}$, a.s.
\begin{eqnarray*}
K'_{s,t}(x) &=& K'_{s,s_n^+}\cdots K'_{t_n,t}(x) \\
&=& K_{s,s_n^+}\cdots K_{t_n,t}(x) \\
&=& K_{s,t}(x).
\end{eqnarray*}

\section{Construction of solutions of $(\hat{E}^v)$ out of a solution of $(E)$.}

Let $\mathcal{W}=(W^i;\; i\in I)$ be as in the previous section. 
Let $K$ be a SFK on $G$.
Assume that $(K,W)$ satisfies the SDE:
For all $s\leq t$, $f\in {\mathcal D}$, $x\in {G}$, a.s.
$$K_{s,t}f(x)=f(x)+\sum_{i\in I} \int_s^tK_{s,u}(1_{E_i}f')(x)W^i(du) + \frac{1}{2}\int_s^tK_{s,u}f''(x)du.$$
Following Lemma 3 \cite{MR50101111}, we prove that $\mathcal F^{W^i}_{s,t}\subset \mathcal F^{K}_{s,t}$ for all $i\in I$ and $s\le t$.
\medskip
In this section, we will prove the following 
\begin{theorem}\label{ktohatk}
For each $v\in V$, there exists a unique (up to modification) SFK $\hat{K}^v$ on $\hat{G}^v$ such that 
\begin{itemize}
\item for all $s\le t$, $\hat{\cF}_{s,t}:=\vee_{v\in V} \cF_{s,t}^{\hat{K}^v}\subset{\cF}_{s,t}^K$, 
\item for all $v\in V$, $(\hat{K}^v,\mathcal{W}^v)$ is a solution to the SDE:
For all $s\leq t$, $\hat{f}\in \hat{\mathcal D}_v$, $\hat{x}\in \hat{G}_v$, a.s.
\begin{equation}\label{equ}
\hat{K}^v_{s,t}\hat{f}(\hat{x})=\hat{f}(\hat{x})+\sum_{i\in I_v} \int_s^t\hat{K}^v_{s,u}(1_{\hat{E}^v_i}\hat{f}')(\hat{x})W^i(du) + \frac{1}{2}\int_s^t\hat{K}^v_{s,u}\hat{f}''(\hat{x})du.
\end{equation}
\end{itemize}
and such that defining for $s\in \RR$ and $x\in G_v$, $\rho^{x,v}_s$ by (\ref{defin}), we have that for all $t>s$, a.s. on the event $\{t<\rho_s^{x,v}\}$, (\ref{holf}) holds.
\end{theorem}

\prf 
Fix $v\in V$. For $s\in\RR$ and $\hat{x}=\hat{e}^v_i(r)\in\hat{G}_v$, if $\hat{x}\in i_v(G_v)$, then denote $x=e_i(r)$ (and we have $\hat{x}=i_v(x)$). Recall the definition of $\tau_s^x$.
Recall also the definition of $A_{s,t}$ from (\ref{A}).
Define the kernel $\hat{K}^{0,v}_{s,t}$ by
\begin{itemize}
\item On $A^c_{s,t}$ :  $\hat{K}^{0,v}_{s,t}(\hat{x})=\delta_{\hat{x}}$.
\item On $A_{s,t}$ : Let $\hat{x}=\hat{e}^v_i(r)$. If $\hat{x}=\hat{e}^v_i(r)\in i_v(G_v)$, $\tau_s^x<t$ and $e_i(r+W^i_{s,\tau_s^x})=v$, define $\hat{K}^{0,v}_{s,t}(\hat{x})=i_v *K_{s,t}(e_i(r))$.
And otherwise, define $\hat{K}^{0,v}_{s,t}(\hat{x})=\delta_{\hat{e}^v_i(r+W^i_{s,t})}$.
\end{itemize}

Now for $n\geq 1$, set
$$\hat{K}^{n,v}_{s,t}=\hat{K}^{0,v}_{s,s^+_n}\hat{K}^{0,v}_{s^+_n,s^+_n+2^{-n}}\dots \hat{K}^{0,v}_{t_n-2^{-n},t_n}\hat{K}^{0,v}_{t_n,t}.$$
if $s_n^+\leq t$ and $\hat{K}^{n,v}_{s,t}=\hat{K}^{0,v}_{s,t}$ if $s_n^+>t$.\\
Define $\Omega^n_{s,t}$, $\Omega_{s,t}$ and $n_{s,t}$ as in Section \ref{cons} and finally set $\hat{K}^{v}_{s,t}=\hat{K}^{n,v}_{s,t}$, where $n=n_{s,t}$ and $\hat{K}^{v}_{s,t}(\hat{x})=\delta_{\hat{x}}$ on $\Omega_{s,t}^c$. Following Sections \ref{cons} and \ref{sat}, we prove that $\hat{K}^v$ is a SFK satisfying  (\ref{equ}). Note that for all $s\le t$, $x\in G_v$, 
$K_{s,t}(x)=i_v^{-1}*\hat{K}^{0,v}_{s,t}(\hat{x}^v)$. Since for all $s\le t$ and $\hat{x}\in \hat{G}_v$, a.s. on $A_{s,t}$, $\hat{K}^{v}_{s,t}=\hat{K}^{0,v}_{s,t}$, the last statement of the Theorem holds.
It remains to remark the uniqueness up to modification, which can be proved in the same manner as for Theorem \ref{hatktok}.
\qed

\medskip
This Theorem implies (ii) of Theorem \ref{klo}.

\medskip

\section{Stochastic flows on star graphs \cite{MR50101010}.}\label{hbvc}
In this section, we overview the content of \cite{MR50101010} where equation $(E)$ on a single star graph has been studied. Let $G=\{0\}\cup\cup_{i\in I} E_i$ be a star graph where $I=\{1,\cdots,n\}$. Assume that $I_{+}=\{i : g_i=0\}=\{1,\cdots,n^+\}$ and  $I_{-}=\{i : d_i=0\}=\{n^++1,\cdots,n\}$ and set $n^-=n-n^+$. To each edge $E_i$, we associate $\alpha^i\in[0,1]$ such that $\sum_{i\in I} \alpha^i=1$. Denote by $e_i$ the orientation of $E_i$ and let $\alpha^+=\sum_{i\in I_+}\alpha^i$, $ \alpha^-=1-\alpha^+$. Let $\alpha=(\alpha^i)_{i\in I}$. In this section, we denote $(E^G_{\alpha})$ simply by $(E)$.

The construction of flows associated to $(E)$ is based on the skew Brownian motion (SBM) flow studied by Burdzy and Kaspi in \cite{MR2094439}. Let $W$ be a real white noise, then the Burdzy-Kaspi (BK) flow $Y$ associated to $W$ and $\beta\in[-1,1]$ is a SFM (see Section $7$ for the definition) solution to 
\begin{equation}\label{equat}
Y_{s,t}(x)=x+W_{s,t}+\beta L_{s,t}(x)
\end{equation}
where $L_{s,t}(x)$ is the local time of $Y_{s,\cdot}(x)$ at time $t$. For $x\in G$, $i\in I$ and $r\in\RR$ such that $x=e_i(r)$, define
$$\tau^x_s=\inf\{t\geq s : e_i(r+W_{s,t})=0\}.$$
\subsection{The case  $\alpha^+\neq\frac{1}{2}$.}
For $k\geq 1$, let $\Delta_k=\big\{u\in [0,1]^k : {\sum_{i=1}^k }u_i=1\big\}$, be the set of probability measures on $\{1,\dots,k\}$. From \cite{MR50101010}, we recall the following
\begin{theorem}\label{jkl} Let $m^+$ and $m^-$ be two probability measures respectively on $\Delta_{n^+}$ and $\Delta_{n^{-}}$ satisfying : $ \forall i\in[1,n^+]$ and $j\in[1,n^-]$, 
$$(+)\ \int_{\Delta_{n^+}} u_i m^+ (du)=\frac{\alpha^i}{\alpha^+},\ \ (-)\ \int_{\Delta_{n^-}}u_j m^- (du)=\frac{\alpha^{j+n^+}}{\alpha^-}.$$
(a) There exists a solution $(K,W)$ on $G$ unique in law such that if $Y$ is the BK flow associated to $W$ and $\beta=2\alpha^+-1$, then for all $s\le t$ in $\RR$, $x\in G$ a.s.
\begin{itemize}
\item[(i)] If $x=e_i(r)$, then $K_{s,t}(x)=\delta_{e_i(r+W_{s,t})}$ on $\{t\le\tau^x_s\}$.
\item [(ii)] On $\{t>\tau^x_s\}$, $K_{s,t}(x)$ is supported on $\{e_i(Y_{s,t}(r)), i\in I_+\}$ if $Y_{s,t}(r)>0$ and on $\{e_i(Y_{s,t}(r)), i\in I_-\}$ if $Y_{s,t}(r)\le 0$.
\item[(iii)] On $\{t>\tau^x_s,\; \pm Y_{s,t}(r)>0\}$, $U^{\pm}_{s,t}(x)=\left({K}_{s,t}(x,E_i),i\in I_{\pm}\right)$ is independent of $W$ and has for law $m^{\pm}$.
\end{itemize}

\noindent (b) For all SFK $K$ such that $(K,W)$ solves $(E)$, there exists a unique pair of measures $(m^+,m^-)$ satisfying conditions $(+)$ and $(-)$ and such that (i), (ii), (iii) above are satisfied.\\

\end{theorem}
Let $U^+=(U^+(i), i\in I_+)$ and $U^-=(U^-(j), j\in I_-)$ be two random variables with values in $\Delta_{n^+}$ and  $\Delta_{n^{-}}$ such that for each $(i,j)\in I_+\times I_-$
$$\PP(U^+(i)=1)=\frac{\alpha^i}{\alpha^+},\ \ \PP(U^-(j)=1)=\frac{\alpha^{j}}{\alpha^-}.$$
Note that all coordinates of $U^{\pm}$ are equal to $0$ expect one coordinate which is equal therefore to $1$.
With $m^+$ and $m^-$ being respectively the laws of $U^+$ and $U^-$, $K_{s,t}(x)=\delta_{\varphi_{s,t}(x)}$ where $\varphi$ is a SFM. The flow $\varphi$ is also the unique SFM solving $(E)$.\\ 
To $U^+=(\frac{\alpha^i}{\alpha^+}, i\in I_+)$ and $U^-=(\frac{\alpha^j}{\alpha^-}, j\in I_-)$, is associated in the same way a Wiener i.e. $\sigma(W)$-measurable solution $K^W$ of $(E)$ which is also the unique (up to modification) Wiener solution to $(E)$.
\subsection{The case $\alpha^+=\frac{1}{2}$.}
In this case $(E)$ admits only one solution $K^W$ which is Wiener, no other solutions can be constructed by adding randomness to $W$. The expression of $K^W$ is the same as the general case with $Y_{s,t}(x)$ replaced by $x+W_{s,t}$. 
\section{Conditional independence : Proof of Proposition \ref{condi}.}
In this section, we assume that for all $i\in I$, $W^i=W$ for some real white noise $W$. Our purpose is to establish Proposition \ref{condi} already proved in \cite{MR50101111} in a very particular case. The main idea was the following : let $(\varphi^+,W)$ and $(\varphi^-,-W)$ be two SFM's solutions to Tanaka's equation:
$$\varphi^{\pm}_{s,t}(x)=x\pm\int_{s}^{t}\text{sgn}(\varphi^{\pm}_{s,u}(x))dW_u.$$
We know that the laws of $(\varphi^+,W)$ and $(\varphi^-,W)$ are unique \cite{MR2235172}. Let $\varphi=(\varphi^+,\varphi^-)$, then if $(\mathcal F^{\varphi}_{s,t})_{s\le t}$ is i.d.i, the law of $\varphi$ is unique. 
An intuitive explanation for this is that $t\mapsto|\varphi^+_{0,t}(0)|=W_t-\inf_{0\leq u\leq t} W_u$ and $t\mapsto|\varphi^-_{0,t}(0)|=\sup_{0\leq u\leq t}W_u-W_t$ do not have common zeros after $0$ so that $\text{sgn}(\varphi^+_{0,t}(0))$ should be independent of $\text{sgn}(\varphi^-_{0,t}(0))$. In the general situation, the previous reflecting Brownian motions are replaced by two SBM's associated to $W$ and distinct skew parameters.

The proof of Proposition \ref{condi} will strongly rely on the following lemma.
\blem \label{sde}
Let $(\beta_1,\beta_2)\in[-1,1]^2$ with $\beta_1\neq\beta_2$ and $|\beta_2-\beta_1|\ge 2\beta_1\beta_2$.
Let $x, y\in\RR$ and let $X, Y$ be solutions of 
$$X_t=x+W_t+\beta_1 L_t(X)\quad \hbox{ and }\quad Y_t=y+W_t+\beta_2 L_t(Y)$$
where $L_t(X)$ and $L_t(Y)$  denote the symmetric local times at $0$ of $X$ and $Y$. If $x\neq y$ or if $x=y=0$, then a.s. for all $t>0$, $X_t\neq Y_t$.
\elem
\prf
Assume first that $x=y=0$. It is straightforward to see that the Lemma holds when $\beta_1\le 0\le\beta_2$. The other cases follow from Theorem 1.4 (i)-(ii) \cite{MR1880238}. 

Assume now that $x\neq y$ : Let $T=\inf\{t>0;\;X_t=Y_t=0\}$. Then necessarily, if $T<\infty$, we have  $X_T=Y_T=0$. So we can conclude using the strong Markov property at time $T$.
\qed
\medskip

\noindent\textbf{Proof of Proposition \ref{condi}.} 
To simplify the notation, for $i\in\{1,2\}$, $G^i$, $\beta_i$, $I^i$, $I^{i,\pm}$ and ${K}^i$ will denote respectively $\hat{G}^{v_i}$, $\beta_{v_i}$, $I_{v_i}$, $I_{v_i}^\pm$ and $\hat{K}^{v_i}$. We will also denote the edges of $G^i$ by $(e^i_j)_{j\in I^i}$ and set $\cF_{s,t}=\cF^{{K}^{1}}_{s,t}\vee\cF^{{K}^{2}}_{s,t}$ for all $s\le t$.\\
It is easy to see that if ${K}^{1}$ and ${K}^{2}$ are independent given $W$, then $(\cF_{s,t})_{s\le t}$ is i.d.i.

Assume now that $(\cF_{s,t})_{s\le t}$ is i.d.i.
For $i\in\{1,2\}$, let $Y^i$ be the BK flow associated to $W$ and $\beta_{i}$. Using the flow property, the stationarity of the flows and the fact that $(\cF_{s,t})_{s\le t}$ is i.d.i., we only need to prove that for all $t>0$, 
\begin{equation}\label{ddd}
{K}_{0,t}^{1}  \; \hbox{ and } \; {K}_{0,t}^{2}\ \text{are independent given}\ W.
\end{equation}
For $n\geq 1$ and $i\in\{1,2\}$, let 
 $({x}^i_j={e}^i_{k^i_j}(r^i_j), 1\le j\le n)$ be $n$ points in ${G}^i$, where $k^i_j\in I^i$ and $r^i_j\in \RR$. Define 
$$\tau^{i}_j=\inf\{u\geq 0 : r^i_j+W_{0,u}=0\}.$$
Proving \eqref{ddd} reduces to prove that 
\begin{equation}\label{dddb}
({K}_{0,t}^{1}(x^1_j))_{1\le j\le n}  \; \hbox{ and } \; ({K}_{0,t}^{2}(x^2_j))_{1\le j\le n}\ \text{are independent given}\ W
\end{equation}
for arbitrary $n$ and $(x^i_j)$.

Note that when $t\le \tau^i_j$, then $K^i_{0,t}(x^i_j)$ is a measurable function of $W$. For $J^1$ and $J^2$ two subsets of $\{1,\dots,n\}$, denote 
$$A_{J^1,J^2}=\bigg\{t>\tau^i_j \hbox{ if and only if } j\in J^i\ \text{for all}\ i=1,2\bigg\}$$
which belongs to $\sigma(W)$. Then proving \eqref{dddb} reduces to check that (for all $J^1$ and $J^2$), given $W$, on  $A_{J^1,J^2}$,
\begin{equation}\label{dddc}
({K}_{0,t}^{1}(x^1_j))_{j\in J^1}  \; \hbox{ and } \; ({K}_{0,t}^{2}(x^2_j))_{j\in J^2}\ \text{are independent.}
\end{equation}
For $j\in J^i$, define
$$g^i_j=\sup\{u\leq t : Y^i_{0,u}(r^i_j)=0\}.$$
Note that a.s. on $A_{J^1,J^2}$, by Lemma \ref{sde}, 
$$\big\{g^1_j : \; j\in J^1\}\cap\big\{g^2_j : \; j\in J^2\}=\emptyset.$$
Let $\mathcal{J}=\{J_k;\; 1\le k\le m\}$ be a partition of $\big(\{1\}\times J^1\big) \cup \big(\{2\}\times J^2\big)$ such that for all $k$, we have $J_k\subset \{i\}\times J^i$ for some $i\in\{1,2\}$ and define the event 
\begin{eqnarray*}
B_{\mathcal{J}}
&=& \left\{g^i_j=g^{i'}_{j'}\ \text{if and only if}\ \exists k \ \text{such that}\ \big((i,j),(i',j')\big)\in J_k\times J_k\right\}\\
&& \bigcap\left\{\forall k<k';\ \text{if} \ \big((i,j),(i',j')\big)\in J_k\times J_{k'} \ \text{then}\ g^i_j < g^{i'}_{j'}\right\}\\
&& \bigcap\;  A_{J^1,J^2}.
\end{eqnarray*}
Then a.s. $\{B_{\mathcal{J}}\}_{\mathcal{J}}$ is a partition of $A_{J^1,J^2}$.\\
\noindent For all $k$, choose $(i_k,j_k)\in J_k$ and denote $g_k=g^{i_k}_{j_k}$.
Let $u_1<\cdots<u_{m-1}$ be fixed dyadic numbers and 
$$C:=C_{u_1,\dots,u_{m-1}}=B_{\mathcal{J}}\cap \{g_k < u_k < g_{k+1};\; 1\le k\le m-1\}.$$
Let $u=u_{m-1}$. On $C$, for all $(i,j)\in J_m$, ${K}^{i}_{0,t}(x^i_j)$ is $\cF_{u,t}\vee \mathcal F^{W}_{0,u}$-measurable. 
Indeed : Fix $(j_+,j_-)\in I^{i,+}\times I^{i,-}$ and define 
$$\begin{array}{ll}
X^{i,j}_u=\begin{cases}
{e}^{i}_{j_+}(Y^i_{0,u}(r^i_j))&\text{if}\  Y^i_{0,u}(r^i_j)\geq 0,\\
{e}^{i}_{j_-}(Y^v_{0,u}(r^i_j))\ &\text {if}\ Y^i_{0,u}(r^i_j)<0,\\
\end{cases}
\end{array}$$
\noindent Then $X^{i,j}_u$ is $\mathcal F^{W}_{0,u}$-measurable, 
$$\inf\{r\geq u : {K}^{i}_{u,r}(X^{i,j}_u)=\delta_0\}=\inf\{r\geq u : Y^{i}_{0,r}(r^i_j)=0\}\le g_m$$
and ${K}^i_{0,t}({x}^i_j)={K}^{i}_{u,t}(X^{i,j}_u)$. 

Moreover, on $C$,  for all $(i,j)\in \cup_{k=1}^{m-1}J_k$, $K^{i}_{0,t}({x}^i_j)$ is ${\cF}_{0,u}\vee \mathcal F^{W}_{u,t}$-measurable (since $Y^{i}_{0,\cdot}(r^i_j)$ do not touch $0$ in the interval $[u,t]$). 
Since $C\in\sigma(W)$, ${\cF}_{0,u}\vee \mathcal F^{W}_{u,t}$ and ${\cF}_{u,t}\vee \mathcal F^{W}_{0,u}$ are independent given $W$, we deduce that on $C$, $\big({K}^{i}_{0,t}({x}^v_j), (i,j)\in J_m\big)$ and $\big({K}^{i}_{0,t}({x}^i_j), (i,j)\in \cup_{k=1}^{m-1}J_k \big)$ are independent given $W$. 
Now an immediate induction permits to show that given $W$, on $C$, \eqref{dddc} is satisfied. Since the dyadic numbers $u_1,\cdots,u_m$ are arbitrary, we deduce that conditionally on $W$, on $B_{\mathcal J}$, \eqref{dddc} is satisfied and finally given $W$, on $A_{J^1,J^2}$, \eqref{dddc} holds.
\qed
\section{Appendix 1: The Burdzy-Kaspi flow}\label{BKF}
In this section, we show how our construction can be simplified on some particular graphs using the BK flow \cite{MR1880238}. Let $(W_{s,t})_{s\le t}$ be a real white noise. For $\beta=\pm 1$, the flow associated to (\ref{equat}) has a simple expression which will be referred as the BK flow. For a fixed $\beta\in]-1,1[$, Burdzy and Kaspi constructed a SFM (see 1.7 in \cite{MR1880238}) satisfying 
\begin{itemize}
\item[(i)] $x\longmapsto Y_{s,t}(x)$ is increasing and c\`adl\`ag for all $s\leq t$ a.s.
\item[(ii)] With probability equal to $1$: $\forall s,x\in \RR, \ \  (Y_{s,t}(x),L_{s,t}(x))$ satisfies (\ref{equat}) and $${L}_{s,t}(x)=\displaystyle{\lim_{\varepsilon \rightarrow 0^+ }}\frac{1}{2\varepsilon }\int_s^t 1_{\{|Y_{s,u}(x)|\leq \varepsilon\}}du.$$
\end{itemize}  
The statement (i) is a consequence of the definition of $Y$ (see also Section 3.1 \cite{MR50101010}) and (ii) can be found in Proposition 1 \cite{MR1880238}. The BK flow satisfies also a strong flow property: 

\begin{proposition}\label{gn}
\begin{itemize}
\item[(1)] Fix $x\in\RR$ and let $S$ be an $(\mathcal F^W_{-\infty,r})_{r\in\RR}$-finite stopping time. Then $Y_{S,S+\cdot}(x)$ is the unique strong solution of the SBM equation with parameter $\beta$ driven by $W_{S,S+\cdot}$. In particular $Y_{S,S+\cdot}$ is independent of $\mathcal F^W_{-\infty,S}$. 
\item[(2)] Let $S\leq T$ be two $(\mathcal F^W_{-\infty,r})_{r\in\RR}$-finite stopping times. Then a.s. for all $u\geq 0, x\in\RR$,
$$Y_{S,T+u}(x)=Y_{T,T+u}\circ Y_{S,T}(x).$$
\end{itemize}
\end{proposition}
\begin{proof}
(1) Let $\mathcal G_t=\mathcal F^W_{-\infty,S+t}, t\geq 0$. Then $Y_{S,S+\cdot}(x)$ is $\mathcal G$-adapted, $W_{S+\cdot}-W_S$ is a $\mathcal G$-Brownian motion and by (ii) above, a.s. $\forall t\geq 0$,
$$Y_{S,S+t}(x)=x+W_{S,S+t}+\beta L_{S,S+t}(x)$$
with $L_{S,S+t}(x)=\displaystyle{\lim_{\varepsilon \rightarrow 0^+ }}\frac{1}{2\varepsilon }\int_0^t 1_{\{|Y_{S,S+u}(x)|\leq \varepsilon\}}du$. Now (1) follows from the main result of \cite{MR606993}.\\
(2) Fix $x\in\RR$. Then by the previous lines for all $y\in\RR$, a.s. $\forall u\geq 0,$ 
$$Y_{T,T+u}(y)=y+W_{T,T+u}+\beta L_{T,T+u}(y).$$
Since $Y_{T,T+\cdot}$ is independent of $\mathcal F^W_{-\infty,T}$ and $Y_{S,T}$ is $\mathcal F^W_{-\infty,T}$ measurable (by the definition of $Y$), it holds that a.s. $\forall u\geq 0,$
$$Y_{T,T+u}(Y_{S,T}(x))=Y_{S,T}(x)+W_{T,T+u}+\beta L_{T,T+u}(Y_{S,T}(x)).$$
Now, set $$Z_r=Y_{S,r}(x) 1_{\{S\leq r\leq T\}}+Y_{T,r}\circ Y_{S,T}(x)  1_{\{r > T\}}.$$
Then, we easily check that a.s. $\forall r\geq S,$
$$Z_r=x+W_{S,r}+\beta \displaystyle{\lim_{\varepsilon \rightarrow 0^+ }}\frac{1}{2\varepsilon }\int_S^r 1_{\{|Z_u|\leq \varepsilon\}}du.$$
By unicity of the solution, a.s. $\forall r\geq T,\ Y_{S,r}(x)=Y_{T,r}\circ Y_{S,T}(x)$. Now using (i) above, (2) holds a.s. for all $x\in\RR$.

\end{proof}
Given a graph $G$ as in Figure 4
we can construct the unique solution $\varphi$ to $(E)$ as follows: To each vertex $v$, let us attach the BK flow $Y^v$ associated to $W$ and $\beta_v:=2\alpha^+_v-1$. If $x=e_i(r)$, define $\varphi_{s,t}(x)=e_i(r+W_{s,t})$ until hitting a vertex point $v_1$ at time $s_1$, then ''define'' $\varphi_{s,t}(x)$ by $Y^{v_1}_{s_1,t}(0)$ until hitting another vertex $v_2$. After $s_2$, $\varphi_{s,t}(x)$ will be ''given by'' $Y^{v_2}_{s_2,t}(0)$ etc. Using Proposition \ref{gn}, we show that $\varphi$ is a SFM. This is the SBM with Barriers flow unique strong solution to the equation
$$X_t=X_0+W_t+\sum_{v\in V}(2\alpha^+_v-1)L^v_t(X)$$

\begin{figure}[!htb]
\centering
\includegraphics[scale=.5]{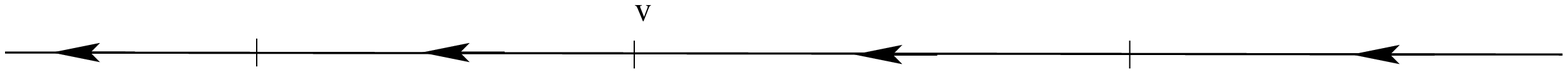}
\caption{SBM with Barriers.}
\end{figure}

%
\medskip
\noindent In \cite{MR50101111}, it is proved that flows solutions of $(E)$ defined on graphs like in Figure 3 can be modified to satisfy strong flow properties similar to Proposition \ref{gn} (2) (see Corollary 2 \cite{MR50101111}). Actually on graphs with arbitrary orientation and transmission parameters and such that each vertex has at most two adjacent edges, we can proceed to a direct construction of ''global'' flows using strong flow properties of ''local'' flows.
\section{Appendix 2: Complement to Section \ref{hbvc}} 
\subsection{A key Lemma on BK flow}
In this section, we work with the same notations as the beginning of Section \ref{hbvc} and let $Y$ be the BK flow as in the previous section associated to $W$ and $\beta:=2\alpha^+-1$. For each $u<v$, let $n$ be first integer such that $]u,v[$ contains a dyadic number of order $n$ and $f(u,v)$ be the smallest dyadic number of order $n$ contained in $]u,v[$ ($f$ is a deterministic machinery which associates to each $(u,v), u<v$ a dyadic number in $]u,v[$). \\
For all $s\leq t,x ,y\in\RR$, with the convention  $\inf \emptyset=+\infty$, set 
\begin{eqnarray}
T^s_{x,y}&=&\inf\{r\geq s,\ Y_{s,r}(x)=Y_{s,r}(y)\},\nonumber\\ 
\tau_{s}(x)&=&\inf\{r\geq s,\ x+W_{s,r}=0\},\nonumber\\
n_{s,t}(x)&=&\inf\bigg\{n\geq 1,\ Y_{s,t}(x-\frac{1}{n})=Y_{s,t}(x+\frac{1}{n})\bigg\}.\nonumber\
\end{eqnarray}
For all $s\leq t, x\in\RR$, let $n=n_{s,t}(x)$ and define
$$v_{s,t}(x)=f(s,T^s_{x-\frac{1}{n},x+\frac{1}{n}})\ \textrm{if} \quad t\geq T^s_{x-\frac{1}{n},x+\frac{1}{n}}$$ 
and $v_{s,t}(x)=0$ otherwise. Now let $(n,v)=\big(n_{s,t}(x),v_{s,t}(x)\big)$ and define      
$$y_{s,t}(x)= f\big(Y_{s,v}(x-\frac{1}{n}),Y_{s,v}(x+\frac{1}{n})\big)\ \textrm{if} \quad  t\geq T^s_{x-\frac{1}{n},x+\frac{1}{n}}$$ 
and  $y_{s,t}(x)=0$ otherwise.
Note that $(s,t,x,\omega)\longmapsto (v_{s,t}(x,\omega),y_{s,t}(x,\omega))$ is measurable and that for all $s<t$, $(v_{s,t},y_{s,t})$ is $\mathcal F^W_{s,t}$-measurable.
\begin{lemma}\label{ko}
Let $s$ and $x$ in $\RR$. Then a.s. for all $t>\tau_s(x)$,  we have
\begin{itemize}
\item[(i)] $n=n_{s,t}(x)<\infty$,
\item[(ii)] $v_{s,t}(x)=f(s,T^s_{x-\frac{1}{n},x+\frac{1}{n}})$ and $y_{s,t}(x)=f(Y_{s,v}(x-\frac{1}{n}),Y_{s,v}(x+\frac{1}{n}))$,
\item[(iii)] $Y_{s,t}(x)=Y_{v,t}(y)$, with $(v,y)=(v_{s,t}(x),y_{s,t}(x))$.
\end{itemize}

\end{lemma}
\begin{proof}
See Lemma 3 in \cite{MR50101010}.
\end{proof}

\subsection{Construction of a flow of mappings}
In this section, we will use the same notations as in the last paragraph and in Section \ref{hbvc} with the assumption $\alpha^+\neq\frac{1}{2}$ if $n\geq 3$. 
Moreover, we set
$$G^+=\{0\}\cup\cup_{i\in I_+} E_i,\ \ G^-=\{0\}\cup\cup_{i\in I_-} E_i.$$

We will review the construction of the unique flow of mappings solving $(E)$ defined on $G$. Let $W$ be a real white noise.
First we will construct $\varphi_{s,\cdot}(x)$ for all $(s,x)\in\QQ\times G_{\QQ}$ where $G_{\QQ}=\{z\in G, |z|\in \QQ_+\}$. Denote this set of points by $(s_i,x_i)_{i\geq0}$ and write $x_i=e_{j_i}(r_i)$ where $r_i\in\RR$ and $j_i\in\{1,\cdots,n\}$. Let $\gamma^+,\gamma^-$ be two independent random variables respectively taking their values in $I_+$ and in $I_-$ and such that for $i\in I_+$ and $j\in I_-$,
$$\PP(\gamma^+=i)=\frac{\alpha^i}{\alpha^+}  \hbox{ and } \PP(\gamma^-=j)=\frac{\alpha^j}{\alpha^-} .$$
We will construct $\varphi_{s_0,\cdot}(x_0)$, then  $\varphi_{s_1,\cdot}(x_1)$ and so on. Let $\mathbb D$ be the set of all dyadic numbers on $\RR$ and $\{(\gamma^+_{r},\gamma^-_{r}), r\in\mathbb D\}$ be a family of independent copies of $(\gamma^+,\gamma^-)$ which is also independent of $W$. If $x=e_i(r)$, recall the definition $\tau^x_s=\tau_s(r)$ where $\tau_s(r)$ is as in the previous paragraph. For $x_0=e_{j_0}(r_0)$, define $\varphi_{s_0,\cdot}(x_0)$ by
$$\varphi_{s_0,t}(x_0)= \begin{cases}
       e_{j_0}(r_0+W_{s_0,t})& \text{if} \ s_0\leq t\leq {\tau}^{x_0}_{s_0}\\
        \ 0 & \text{if} \ t>{\tau}^{x_0}_{s_0}, Y_{s_0,t}(r_0)=0\\ 
        e_{h}(Y_{s_0,t}(r_0)),\ \  & \text{if} \ \gamma^{+}_{r}=h,\ t>{\tau}^{x_0}_{s_0}, Y_{s_0,t}(r_0)>0\\
        e_{h}(Y_{s_0,t}(r_0)),\ \  & \text{if} \ \gamma^{-}_{r}=h,\ t>{\tau}^{x_0}_{s_0}, Y_{s_0,t}(r_0)<0
       \end{cases}$$ 
where $r=f(u,v)$ and $u,v$ are respectively the last zero before $t$ and the first zero after $t$ of $Y_{s_0,\cdot}(r)$ (well defined when $Y_{s_0,t}(r_0)\neq 0$). Now, suppose that $\varphi_{s_0,\cdot}(x_0),\cdots,\varphi_{s_{q-1},\cdot}(x_{q-1})$ are defined and let $\{(\gamma^+_r,\gamma^-_r), r\in\mathbb D\}$ be a new family of independent copies of $(\gamma^+,\gamma^-)$ (that is independent of all vectors $(\gamma^+,\gamma^-)$ used until $q-1$ and independent also of $W$). Let $$t_0=\inf\left\{u\geq s_{q} : Y_{s_{q},u}(r_q)\in \{Y_{s_i,u}(r_i), i\in [0,q-1]\}\right\}.$$
Since $t_0<\infty$, let $i\in[0,q-1]$ and $(s_i,r_i)$ such that $Y_{s_q,t_0}(r_q)=Y_{s_i,t_0}(r_i)$. Now define $\varphi_{s_q,\cdot}(x_q)$ by 
$$\varphi_{s_q,t}(x_q)= \begin{cases}
       e_{j_q}(r_q+W_{s_q,t})& \text{if} \ s_q\leq t\leq {\tau}^{x_q}_{s_q}\\
        \ 0 & \text{if} \ {\tau}^{x_q}_{s_q}<t<t_0, Y_{s_q,t}(r_q)=0\\ 
        e_{h}(Y_{s_0,t}(r_0)),\ \  & \text{if} \ \gamma^{+}_{r}=h,\ {\tau}^{x_q}_{s_q}<t<t_0, Y_{s_q,t}(r_q)>0\\
        e_{h}(Y_{s_q,t}(r_q)),\ \  & \text{if} \ \gamma^{-}_{r}=h,\ {\tau}^{x_q}_{s_q}<t<t_0, Y_{s_q,t}(r_q)<0\\
        \varphi_{s_i,t}(x_i) & \text{if} \ t\geq t_0\\
       \end{cases}$$ 
where $r$ is defined as in $\varphi_{s_0,\cdot}(x_0)$ (from the skew Brownian motion $Y_{s_q,\cdot}(r_q)$). In this way, we construct $(\varphi_{s_i,\cdot}(x_i))_{i\geq 0}$.\\
\noindent\textbf{Extension.}
Now we will define entirely $\varphi$. Let $s\leq t$, $x\in G$ such that $(s,x)\notin\QQ\times G_{\QQ}$. If $x=e_i(r), s\leq t\leq \tau^x_s$, define $\varphi_{s,t}(x)=e_i(r+W_{s,t})$. If $t>\tau^x_s$, let $m$ be the first nonzero integer such that $Y_{s,t}(r-\frac{1}{m})=Y_{s,t}(r+\frac{1}{m})$ (when $m$ does not exist we give an arbitrary definition to $\varphi_{s,t}(x)$). Then consider the dyadic numbers 
\begin{equation}\label{nbvc}
v=f\left(s,T^s_{r-\frac{1}{m},r+\frac{1}{m}}\right),\ \ r'=f\left(Y_{s,v}(r-\frac{1}{m}),Y_{s,v}(r+\frac{1}{m})\right)
\end{equation}
and finally set $\varphi_{s,t}(x)=\varphi_{v,t}(z)$ where
\begin{equation}\label{pok}
z=e_1(r')\ \text{if}\ \ r'\geq 0\ \text{and}\ \ z=e_{n^++1}(r')\ \text{if}\ \ r'<0.
\end{equation}
Note that $\varphi_{s,t}(x,\omega)$ is measurable with respect to $(s,t,x,\omega)$.\\
By Lemma 3 \cite{MR50101010}, for a ''typical'' $(s,x)$ a.s. for all $t>\tau^x_s$, $m$ is finite.  Note also that : for all $s\le t$, $x=e_i(r)\in G$ a.s. 
\begin{equation}\label{plkm}
|\varphi_{s,t}(x)|=|Y_{s,t}(r)|\ \text{and}\ \varphi_{s,t}(x)\in G^{\pm}\Leftrightarrow \ \pm Y_{s,t}(r)\geq 0.
\end{equation}
This is clear when $(s,x)\in\QQ\times G_{\QQ}$ and remains true for all $s, t$ and $x$ by Lemma \ref{ko} (iii). The independence of increments of $\varphi$ is clear and the stationarity comes from the fact that for all $s\le t$ and $x=e_i(r)\in G$ (even when $(s,x)\in\QQ\times G_{\QQ}$), if $v$ and $r'$ are defined by (\ref{nbvc}), then on the event $\{t>\tau^x_s\}$, a.s. $\varphi_{s,t}(x)=\varphi_{v,t}(z)$ with $z$ given by (\ref{pok}).\\
Writing Freidlin-Sheu formula (see Theorem 3 in \cite{MR50101010}) for the Walsh Brownian motion $t\mapsto\varphi_{s,s+t}(x)$ and using (\ref{plkm}), we see that $\varphi$ solves $(E)$.\\
The flow $\varphi$ is the unique SFM solving $(E)$ in our case. When $\alpha^+=\frac{1}{2}$, the BK flow is the trivial flow $x+W_{s,t}$ which is non coalescing. The above construction cannot be applied if $n\geq 3$, no flow of mappings solving $(E)$ can be constructed in this case.
\begin{remark} Recall the text after Theorem \ref{jkl} (the SFM case). Then $(U^+,U^-)$ can be identified with a couple $(\gamma^+,\gamma^-)$ with law as described above. We have seen that working directly with $(\gamma^+,\gamma^-)$ makes the construction more clear.

\end{remark}

\subsection{The other solutions.}
Suppose $\alpha^+\neq\frac{1}{2}$ and let $m^+$ and $m^-$ be two probability measures as in Theorem \ref{jkl}. Then, to $(m^+,m^-)$ is associated a SFK $K$ solution of $(E)$ constructed similarly to $\varphi$. Let $U^+=(U^+(i))_{i\in I_+}$ and $U^-=(U^-(j))_{j\in I_-}$ be two independent random variables with values in $[0,1]^{n^+}$ and $[0,1]^{n^-}$ such that
$$U^+\overset{law}{=}m^+,\ \ U^-\overset{law}{=}m^-.$$
In particular a.s. $\sum_{i\in I_+}U^{+}(i)=\sum_{i\in I_-}U^{-}(j)=1$. Let $\{(U^+_r,U^-_r), r\in\mathbb D\}$ be a family of independent copies of $(U^+,U^-)$ which is independent of $W$. 
Then define
$$K_{s_0,t}(x_0)= \begin{cases}
     \delta_{e_{j_0}(r_0+W_{s_0,t})}& \text{if} \ s_0\leq t\leq {\tau}^{x_0}_{s_0}\\
       \ \delta_0 & \text{if} \quad t>{\tau}^{x_0}_{s_0}, \;Y_{s_0,t}(r_0)=0\\ 
     \displaystyle{\sum_{i\in I_+}}U^{+}_{r}(i) \delta_{e_{i}(Y_{s_0,t}(r_0))},\ \  & \text{if} \quad t>{\tau}^{x_0}_{s_0}, \; Y_{s_0,t}(r_0)>0\\
         \displaystyle{\sum_{j\in I_-}}U^{-}_{r}(j)\delta_{e_{j}(Y_{s_0,t}(r_0))},\ \  & \text{if} \quad t>{\tau}^{x_0}_{s_0},\; Y_{s_0,t}(r_0)<0
       \end{cases}$$ 
       where  $U_r^{+}=(U^{+}_r(i))_{i\in I_+}, U_r^{-}=(U_r^{-}(j))_{j\in I_-}$ and $r$ is the same as in the definition of $\varphi_{s_0,\cdot}(x_0)$. Now $K$ is constructed following the same steps as $\varphi$.\\
\bibliographystyle{plain}
\bibliography{Bil5}

\begin{thebibliography}{10}

\bibitem{MR1022917}
M.~Barlow, J.~Pitman, and M.~Yor.
\newblock On {W}alsh's {B}rownian motions.
\newblock In {\em S\'eminaire de {P}robabilit\'es, {XXIII}}, volume 1372 of
  {\em Lecture Notes in Math.}, pages 275--293. Springer, Berlin, 1989.

\bibitem{MR2094439}
K.~Burdzy and H.~Kaspi.
\newblock Lenses in skew {B}rownian flow.
\newblock {\em Ann. Probab.}, 32(4):3085--3115, 2004.

\bibitem{MR1880238}
Krzysztof Burdzy and Zhen-Qing Chen.
\newblock Local time flow related to skew {B}rownian motion.
\newblock {\em Ann. Probab.}, 29(4):1693--1715, 2001.

\bibitem{MR0043202}
M.~Freidlin and H.~Pavlopoulos.
\newblock On a {S}tochastic {M}odel for {M}oisture {B}udget in an {E}ulerian
  {A}tmospheric {C}olumn.
\newblock {\em Environmetrics, 8:425-440}, 1997.

\bibitem{MR1743769}
M.~Freidlin and S.~Sheu.
\newblock Diffusion {p}rocesses on {g}raphs: {s}tochastic {d}ifferential
  {e}quations, {l}arge {d}eviation {p}rinciple.
\newblock {\em Probab. Theory Related Fields}, 116(2):181--220, 2000.

\bibitem{MR1399081}
Mark Freidlin.
\newblock {\em Markov processes and differential equations: asymptotic
  problems}.
\newblock Lectures in Mathematics ETH Z\"urich. Birkh\"auser Verlag, Basel,
  1996.

\bibitem{MR1743222}
M.I. Freidlin and A.D. Wentzell.
\newblock {D}iffusion {p}rocesses on {g}raphs and the {A}veraging {P}rinciple.
\newblock {\em The Annals of Probability 21:2215-2245}, 1993.

\bibitem{MR50101010}
H.~Hajri.
\newblock Stochastic flows related to {W}alsh {B}rownian motion.
\newblock {\em Electronic journal of probability 16, 1563-1599}, 2011.

\bibitem{MR50101111}
H.~Hajri and O.~Raimond.
\newblock Tanaka's equation on the circle and stochastic flows.
\newblock {\em To appear in ALEA Lat. Am. J. Probab. Math. Stat. Available via
  http://hal.archives-ouvertes.fr/docs/00/68/02/40/PDF/Cercle.pdf}, 2012.

\bibitem{MR606993}
J.M. Harrison and L.A. Shepp.
\newblock On skew {B}rownian motion.
\newblock {\em Ann. Probab.}, 9(2):309--313, 1981.

\bibitem{MR2905755}
V.~Kostrykin, J.~Potthoff, and R.~Schrader.
\newblock Brownian motions on metric graphs {I} - {D}efinition, {F}eller
  {P}roperty, and {G}enerators.
\newblock {\em Availbale via : http://arxiv.org/pdf/1012.0733.pdf}, 2010.

\bibitem{MR2905744}
V.~Kostrykin, J.~Potthoff, and R.~Schrader.
\newblock Brownian motions on metric graphs {II} - {C}onstruction of {B}rownian
  {M}otions on {S}ingle {V}ertex {G}raphs.
\newblock {\em Available via : http://arxiv.org/pdf/1012.0737.pdf}, 2010.

\bibitem{MR2905733}
V.~Kostrykin, J.~Potthoff, and R.~Schrader.
\newblock Brownian motions on metric graphs {III} - {C}onstruction: {G}eneral
  {M}etric {G}raphs.
\newblock {\em Available via : http://arxiv.org/pdf/1012.0739.pdf}, 2010.

\bibitem{MR2905788}
V.~Kostrykin, J.~Potthoff, and R.~Schrader.
\newblock Brownian motions on metric graphs.
\newblock {\em J. Math. Phys. 53}, 2012.

\bibitem{MR2060298}
Y.~Le~Jan and O.~Raimond.
\newblock Flows, coalescence and noise.
\newblock {\em Ann. Probab.}, 32(2):1247--1315, 2004.

\bibitem{MR2235172}
Y.~Le~Jan and O.~Raimond.
\newblock Flows associated to {T}anaka's {SDE}.
\newblock {\em ALEA Lat. Am. J. Probab. Math. Stat.}, 1:21--34, 2006.

\bibitem{MR0047702}
S.~Nicaise.
\newblock Some {R}esults on {S}pectral {T}heory over {N}etworks, {A}pplied to
  {N}erve {I}mpulse {T}ransmission.
\newblock {\em LNin Math, 1171:532-541}, 1985.

\end{thebibliography}
\end{document}